\documentclass[reqno]{amsart}

\usepackage{times}
\usepackage{mathrsfs}
\usepackage{amsmath}
\usepackage{amsthm}
\usepackage{amsfonts}

\newtheorem{thm}{Theorem}[section]

\newtheorem{cor}[thm]{Corollary}
\newtheorem{prop}[thm]{Proposition}

\newtheorem{rem}[thm]{Remark}

\newtheorem{deff}[thm]{Definition}

\numberwithin{equation}{section}

\newcommand{\Wqbb}{\mathbb{W}_{p,\mathcal{B}}}
\newcommand{\Wqb}{W_{p,\mathcal{B}}}
\newcommand{\R}{\mathbb{R}}

\newcommand{\E}{\mathbb{E}}

\newcommand{\Om}{\Omega}

\newcommand{\rd}{\mathrm{d}}
\newcommand{\divv}{\mathrm{div}}

\newcommand{\h}{h}
\newcommand{\g}{g}
\newcommand{\bqn}{\begin{equation}}
\newcommand{\eqn}{\end{equation}}
\newcommand{\bqnn}{\begin{equation*}}
\newcommand{\eqnn}{\end{equation*}}
\newcommand{\bear}{\begin{eqnarray}} 
\newcommand{\eear}{\end{eqnarray}} 
\newcommand{\bean}{\begin{eqnarray*}} 
\newcommand{\eean}{\end{eqnarray*}} 
\newcommand{\bs}{\begin{split}}
\newcommand{\es}{\end{split}}

\title[Age-Dependent Equations with Non-Linear Diffusion]
{Age-Dependent Equations with Non-Linear Diffusion}

\author[Ch. Walker]{Christoph Walker}

\address{Leibniz Universit\"at Hannover, Institut f\"ur Angewandte Mathematik, Welfengarten 1, D--30167 Hannover, Germany.}
\email{walker@ifam.uni-hannover.de}

\begin{document}

\begin{abstract}
We consider the well-posedness of models involving age structure and non-linear diffusion. Such problems arise in the study of population dynamics. It is shown how diffusion and age boundary conditions can be treated that depend non-linearly and possibly non-locally on the density itself. The abstract approach is depicted with examples.
\end{abstract}

\keywords{Age structure, non-linear diffusion, population models, evolution systems.
\\
{\it Mathematics Subject Classifications (2000)}: 35M10, 35K90, 92D25.}

\maketitle

%%%%%%%%%%%%%%%%%%%%%%%%%%%%%%%%%%%
\section{Introduction}
%%%%%%%%%%%%%%%%%%%%%%%%%%%%%%%%%%%

\noindent We consider abstract non-linear problems that naturally arise in the study of the dynamics of populations structured by age and spatial position (e.g. see \cite{WebbSpringer} and the references therein). More precisely, we are interested in Banach-space-valued solutions to equations of the form
\begin{align}
\partial_t u\,+\, \partial_au &\, =\,     -A[\bar{u}](t)\,u\,-\,m\big(t,a,\bar{u}(t)\big)\, u \ ,\!\!\!\!\!\!\!\!\!\!\!\!\!\!\!\!\!\!\!\!\!\!\!& t>0\ ,\quad a>0\ ,\label{1}\\ 
u(t,0)&\, =\, B[u](t)\ ,& t>0\ ,\label{2}\\
u(0,a)&\, =\,  u^0(a)\ ,& a>0\ ,\label{3}\\
\bar{u}(t)&\, =\, \int_0^\infty u(t,a)\, \h(a)\,\rd a\ ,& t>0\ .\label{4}
\end{align}
The function $u=u(t,a)$ usually represents the population density of a certain specie at time $t>0$ and age $a>0$, so that $\bar{u}(t)$ in equation \eqref{4} is the (weighted) total population independent of age. The operator $A[\bar{u}](t)$ in equation \eqref{1} acts for a fixed function $\bar{u}$ and time $t$ as a linear (and unbounded) operator on a Banach space $E_0$. In concrete applications, $A[\bar{u}](t)$ plays the role of non-linear diffusion. Equation \eqref{2} reflects the age-boundary conditions depending on the biological context.

The main features of equations \eqref{1}-\eqref{4} are the non-linear dependence of the operators $A$ and $B$ on the (total) density $u$. While a great part of the research so far focused on linear diffusion, it is the aim of this paper to present an approach in an abstract setting giving a framework for a larger class of problems of the form \eqref{1}-\eqref{4}. This will not only provide us with some flexibility in choosing the underlying functional spaces in concrete applications, but also allows us to consider non-linear diffusion and age-boundary conditions that may depend locally or possibly non-locally with respect to time on the density $u$.
The approach applies to general second order time-dependent elliptic operators on a smooth domain $\Om\subset\R^n$, e.g. to operators of the form
$$
A[\bar{u}](t)w\,=\,-\nabla_x\cdot\big(D\big(\Phi(\bar{u})(t)\big)\nabla_x w\big)
$$
for some smooth function $D$ with $D(z)\ge d_0>0$, $z\in \R$, subject to suitable boundary conditions on $\partial\Om$. Here, the function $\Phi$ is a suitable function merely depending on $\bar{u}([0,t])$, in particular, $\Phi(\bar{u})(t)=\bar{u}(t)$ is possible. A reasonable choice is then $E_0=L_p(\Om)$ with $p\in [1,\infty)$. As for the non-linear age-boundary condition \eqref{2}, the operator $B$ may also depend locally or non-locally on the density $u$. For instance, we may incorporate birth boundary conditions of the form
$$
B[u](t) =\int_0^\infty b\big(t,a,\bar{u}(t)\big)\,u(t,a)\,\rd a
$$
with some suitable birth modulus $b$ (e.g., see \cite{WebbSpringer}), or also age boundary conditions with history-dependent birth function of the form
$$
B[u](t) =\int_0^\infty  b\left(t,a,\int_{-\tau}^0 \bar{u}(t+\sigma)\rd \sigma\right)\,  u(t,a)\ \rd a
$$
as contemplated in \cite{DiBlasio}, where $\tau>0$ is the maximal delay. We refer to our examples in Section \ref{examples}.\\

In the next section, Section \ref{main result}, we first list our assumptions and introduce the notion of a (generalized) solution to \eqref{1}-\eqref{4} before stating our main results on the well-posedness of \eqref{1}-\eqref{4}. This section is then supplemented with further properties of the solution such as regularity, positivity, and global existence. The proof of the main result, Theorem \ref{T}, will be performed in Section \ref{proof of theorem T}, while the proofs of the additional properties will be given in Section \ref{further proofs}. Finally, in Section \ref{examples} we briefly indicate how to apply these results in problems occurring in different situations of population dynamics. 

We shall point out that other notions of solutions and other solution methods for age structured equations with linear diffusion were also introduced in literature, e.g. using integrated semigroups (see \cite{MagalShigui,MagalThieme} and the references therein) or using perturbation arguments (see \cite{NickelRhandi,Rhandi,RhandiSchnaubelt}). For a similar approach as in the present paper we refer to \cite{KunischSchappacherWebb,LaurencotWalkerOpus1,WalkerDIE,WalkerEJAM,WebbSpringer}. We also refer to \cite{BusenbergIannelli1,BusenbergIannelli2,BusenbergIannelli3,LaurencotWalkerOpus2} for other approaches to age structured equations with non-linear diffusion.

\section{Main Results}\label{main result}

In the following, we assume that $E_1$ and $E_0$ are Banach spaces such that $E_1$ is densely and continuously embedded in $E_0$. Furthermore, $(\cdot,\cdot)_\theta$ is for each $\theta\in (0,1)$ an admissible interpolation functor, that is, $E_1$ is densely embedded in each $E_\theta:=(E_0,E_1)_\theta$. Let $\mathcal{L}(E_1,E_0)$ denote the space of all bounded and linear operators from $E_1$ into $E_0$ equipped with the usual uniform operator norm. Given $\omega>0$ and $\kappa\ge 1$ we write $$A\in\mathcal{H}(E_1,E_0;\kappa,\omega)$$ provided $A\in\mathcal{L}(E_1,E_0)$ is such that $\omega+A$ is an isomorphism from $E_1$ onto $E_0$ and satisfies 
$$
\frac{1}{\kappa}\,\le\,\frac{\|(\lambda+A)u\|_{\mathcal{L}(E_1,E_0)}}{\vert\lambda\vert \,\| u\|_{E_0}+\|u\|_{E_1}}\,\le \, \kappa\ ,\quad Re\, \lambda\ge \omega\ ,\quad u\in E_1\setminus\{0\}\ .
$$
We set
$$
\mathcal{H}(E_1,E_0):=\bigcup_{ \substack{\kappa\ge 1 \\ \omega>0}} \mathcal{H}(E_1,E_0;\kappa,\omega)\ ,
$$
which (equipped with the topology induced by the uniform operator norm) is an open subset of $\mathcal{L}(E_1,E_0)$. It is well known that $A\in \mathcal{H}(E_1,E_0)$ if and only if $-A$, considered as a linear operator in $E_0$ with domain $E_1$, is the generator of a strongly continuous analytic semigroup on $E_0$, e.g. see \cite{LQPP}.

Next, we fix a function $\g\in L_{\infty,loc}^+(\R^+)$ satisfying
\bqn\label{a1111}
0\,<\,g_0\,\le\, \g (a+b)\,\le\, g_1\,\g (a)\,\g (b)\ ,\quad a,b>0\ ,
\eqn
for some numbers $g_j>0$, and we introduce the Banach space $$
\mathbb{E}_\theta:=L_1\big(\R^+,E_\theta,\g(a)\rd a\big)\ .
$$ 
If $T>0$, we put $I_T:=[0,T]$.
Given a function $u\in\mathbb{E}_0^{I_T}$, we simply write $u(t,a)$ for $t\in I_T$ and $a> 0$ instead of $u(t)(a)$. For an interval $J$ we set $\dot{J}:=J\setminus\{0\}$.

For $\sigma\in \R$ and $\gamma\in [0,1]$ let $C_\sigma((0,T],\E_\gamma)$ be the space of all continuous functions $v:(0,T]\rightarrow \E_\gamma$ such that $t\mapsto t^\sigma v(t)$ stays bounded in the norm of $\E_\gamma$.\\

Throughout we suppose that there exists a number $\alpha\in [0,1)$ such that the following assumptions hold:\\
\begin{itemize}

\item[$(A_1)$] The function $\h\in L_{\infty,loc}^+(\R^+)$ satisfies  $\varlimsup\limits_{a\rightarrow\infty}\frac{\h(a)}{\g (a)}<\infty$, and there exists $\zeta>0$ such that for each $T>0$ there is $c(T)>0$ with 
$$
\vert\h(t+a)-\h(t_*+a)\vert\,\le\,c(T)\,\g (a)\, \vert t-t_*\vert^{\zeta}\ ,\quad 0\le t,t_*\le T\ ,\quad a\ge 0\ .$$

\item[$(A_2)$] Given $T_0, R>0$ and $\theta\in (0,1)$ there are numbers $\rho\in (0,1)$, $\omega>0$, $\kappa\ge 1$, \mbox{$\sigma\in \R$}, and $c_0>0$ (depending possibly on $\theta$, $T_0$, and $R$) such that for each $T\in (0,T_0]$ the operator $A=\big[\bar{u}\mapsto A[\bar{u}]\big]$ maps $C^\theta(I_T,E_\alpha)$ into $C^\rho(I_T,\mathcal{L}(E_1,E_0))$
and satisfies
\bqn\label{t5}
\sigma +A[\bar{u}]\in C\big( I_T,\mathcal{H}(E_1,E_0;\kappa,\omega) \big)\ ,
\qquad \left\|A[\bar{u}]\right\|_{C^\rho(I_T,\mathcal{L}(E_1,E_0))}\,\le\,c_0\ ,
\eqn
and
\bqn\label{t6}
\left\|A[\bar{u}]-A[\bar{u}_*]\right\|_{C(I_T,\mathcal{L}(E_1,E_0))}\,\le\,c_0\, \|\bar{u}-\bar{u}_*\|_{C(I_T,E_\alpha)}
\eqn
for all $\bar{u},\bar{u}_*\in C^\theta(I_T,E_\alpha)$ with $\|\bar{u}\|_{C^\theta(I_T,E_\alpha)}\le R$ and $\|\bar{u}_*\|_{C^\theta(I_T,E_\alpha)}\le R$.
Moreover, if $0<T<S$ and $\bar{u},\bar{u}_*\in C(I_S,E_{\alpha})$ with $u\big\vert_{I_T}=u_*\big\vert_{I_T}$, then $A[\bar{u}]\big\vert_{I_T}=A[\bar{u}_*]\big\vert_{I_T}$. 

\item[$(A_3)$] There exists $\mu>0$ such that, for $0<T\le T_0$ and $R>0$, the function $B$ maps $C(I_T,\mathbb{E}_{\alpha})$ into $C(I_T, E_\mu)$, and there exists some $c_0=c_0(T_0,R)>0$ such that
\bqn\label{B}
\|B[u]-B[u_*]\|_{C(I_T,E_\mu)}\,\le\, c_0\, \|u-u_*\|_{C(I_T,\mathbb{E}_\alpha)}
\eqn
provided that $u,u_*\in C(I_T,\mathbb{E}_\alpha)$ with $\|u\|_{C(I_T,\mathbb{E}_{\alpha})}\le R$ and $\|u_*\|_{C(I_T,\mathbb{E}_{\alpha})}\le R$. In addition, if $0<T<S$ and $u,u_*\in C(I_S,\mathbb{E}_{\alpha})$ with $u\big\vert_{I_T}=u_*\big\vert_{I_T}$, then $B[u]\big\vert_{I_T}=B[u_*]\big\vert_{I_T}$. 

\item[$(A_4)$] The function $m\in C(\R^+\times\R^+\times E_{\alpha},\R)$ is such that, given $T>0$ and $R>0$, there exists $c_0=c_0(T,R)>0$ with
$$
\vert m(t,a,\bar{u})-m(t,a,\bar{u}_*)\vert\,\le\, c_0\, \|\bar{u}-\bar{u}_*\|_{E_{\alpha}}
$$
and
\bqn\label{mm}
 e^{-m(t,a,\bar{u})}\,\le\, c_0
\eqn
for $t\in I_T$, $a> 0$, and $ \|\bar{u}\|_{E_{\alpha}},  \|\bar{u}_*\|_{E_{\alpha}}\le R$.\\
\end{itemize}

The latter assumptions in $(A_2)$ and $(A_3)$ guarantee that equations \mbox{\eqref{1}-\eqref{4}} pose a proper time evolution problem, that is, the solution depends at each time $t$ only on the past but not on the future. In Section \ref{examples} we will give concrete examples for operators $A$ and $B$ satisfying $(A_2)$ and $(A_3)$, respectively. In particular, it will be shown that if $A$ depends locally with respect to time on $\bar{u}$ and if $E_1$ is compactly embedded in $E_0$, then $(A_2)$ is rather easy to verify in applications (see Proposition \ref{ex1} and Corollary~\ref{C22}). Introducing the function $\g$ in the definition of the spaces $\E_\theta$ allows to give a meaning to \eqref{4} for $u\in \E_0^{I_T}$ in view of assumption $(A_1)$. Also note that \eqref{mm} is trivially satisfied if $m$ is non-negative or bounded.\\

In order to introduce the notion of a solution to \eqref{1}-\eqref{4}, we first observe that if $\bar{u}:I_T\rightarrow E_\alpha$ is H\"older continuous,
then \cite[II.Cor.4.4.2]{LQPP} and \eqref{t5} ensure that $-A[\bar{u}]$ generates a unique evolution system $U_{A[\bar{u}]}(t,s)$, $0\le s\le t\le T$, on $E_0$. 

\begin{deff}
A function $u\in C(J,\E_\alpha)$ is a {\it generalized solution} to \eqref{1}-\eqref{4} on an interval $J$ provided that
\begin{itemize}
\item[(i)] $\bar{u}:J\rightarrow E_\alpha$ is H\"older continuous,
\item[(ii)] $u$ satisfies 
    \bqnn
     u(t,a)\, =\, \left\{ \begin{aligned}
    &e^{-\int_0^am_{\bar{u}}(s+t-a,s)\rd s}\, U_{A[\bar{u}]}(t,t-a)\, B[u](t-a)\ ,\quad & 0\le a< t&\ ,\\
    & e^{-\int_0^t m_{\bar{u}}(s,s+a-t)\rd s}\, U_{A[\bar{u}]}(t,0)\, u^0(a-t)\ ,& 0\le t<a&\ ,
    \end{aligned}
   \right.
    \eqnn
    for $t\in J$ and $a>0$, where $m_{\bar{u}}(t,a):=m\big(t,a,\bar{u}(t)\big)$. 
\end{itemize}
\end{deff}    
The notion of a generalized solution is derived by integrating \eqref{1}-\eqref{4} formally along characteristics. Proposition \ref{P1} below gives more details regarding further regularity of generalized solutions.\\

We first state an existence and uniqueness result for generalized solutions to \eqref{1}-\eqref{4}.

\begin{thm}\label{T}
Suppose $(A_1)-(A_4)$ with $\alpha\in [0,1)$ and let $0\le \alpha<\beta\le 1$. Then, given $u^0\in\E_\beta$, there exists a unique maximal generalized solution $u:=u(\cdot;u^0)$ to \eqref{1}-\eqref{4} on an interval $J:=J(u^0)$ with $$u\in C(J,\E_\beta)\cap C_{\upsilon-\beta}((0,T],\E_\upsilon)\ ,\quad \beta\le \upsilon\le 1\ , \quad T\in\dot{J}\ .$$
Moreover, 
$$
\int_0^\infty u(\cdot,a)\,\tilde{\h}(a)\, \rd a \in C^{\zeta\wedge (\beta-\gamma)}(J,E_\gamma)
$$ 
for $\gamma\in [\alpha,\beta)$ and any function $\tilde{\h}$ satisfying $(A_1)$. In addition,
\bqn\label{dd}
\int_0^\infty u(\cdot,a)\,\rd a \in C^1(\dot{J},E_0)\cap C(\dot{J},E_1)\ .
\eqn
The maximal interval of existence, $J$, is open in $\R^+$, and if 
\bqn\label{global}
\sup_{t\in J\cap [0,T]} \| u(t,\cdot)\|_{\E_\beta}\,<\,\infty\ ,\quad T>0\ ,
\eqn
then the solution exists globally, that is, $J=\R^+$.
\end{thm}

A proof of this theorem will be given in Section \ref{proof of theorem T}. Before providing more properties of the generalized solution, we shall emphasize that the regularity assumptions on the operators $A$ and $B$ in $(A_2)$ and $(A_3)$ are imposed to overcome the difficulties induced by the quasi-linear structure of $A=A[\bar{u}]$. Indeed,  in the case of ``linear diffusion'', that is, if $A=A(t)$ depends possibly on time but is independent of $\bar{u}$, less assumptions are required. For simplicity, we state the following remark for a function $m=m(t,a)$ that is independent of $\bar{u}$.

\begin{rem}\label{R0}
Suppose that $A\in C^\rho (\R^+,\mathcal{H}(E_1,E_0))$ for some $\rho>0$, and for each $T>0$ let there be numbers $0\le \alpha\le \beta \le 1$ with $(\alpha,\beta)\not=(0,1)$ such that the function \mbox{$B:C(I_T,\mathbb{E}_\beta)\rightarrow C(I_T,E_\alpha)$} is uniformly Lipschitz continuous on bounded sets and satisfies
$B[u]\vert_{I_T}=B[u_*]\vert_{I_T}$ for $0<T<S$, $u,u_*\in C(I_S,\mathbb{E}_{\beta})$, and $u\vert_{I_T}=u_*\vert_{I_T}$. If $m\in C(\R^+\times\R^+)$ is bounded, then the problem
\bear
\partial_t u\,+\, \partial_au &=& -A(t)\,u\,-\,m(t,a)\, u \ ,\quad t>0\ ,\quad a>0\ ,\notag\\
u(t,0)&=&B[u](t)\ ,\quad t>0\ ,\notag\\
u(0,a)&=& u^0(a)\ ,\quad a>0\ ,\notag
\eear
admits for each $u^0\in\E_\beta$ a unique maximal generalized solution $u\in C(J,\E_\beta)$, which exists globally if \eqref{global} holds.
\end{rem}

\noindent A proof of this remark follows along the lines of the proof of Theorem \ref{T} and we thus omit details.\\

\noindent We now give additional properties of the generalized solution. For the rest of this section, we suppose the assumptions of Theorem \ref{T}, and we fix $u^0\in\E_\beta$ and let $u=u(\cdot;u^0)\in C(J,\E_\beta)\cap C_{\upsilon-\beta}((0,T],\E_\upsilon)$ for $T\in \dot{J}$ and $\beta\le \upsilon\le 1$ denote the unique maximal generalized solution to \eqref{1}-\eqref{4} on $J=J(u^0)$ corresponding to $u^0$. \\

First we mention that the solution depends continuously on the initial value $u^0\in \E_\beta$. More precisely, we have

\begin{cor}\label{C1}
Given $u^0\in \E_\beta$ there exists $\delta>0$ and $T=T(u^0)>0$ such that $J(u_*^0)\supset [0,T]$ for every $u_*^0\in\E_\beta$ with $\left\| u^0-u_*^0\right\|_{\E_\beta}\le \delta$. Moreover, $u(\cdot;u_*^0)\rightarrow u(\cdot;u^0)$ in $C([0,T],\E_\beta)$ as $u_*^0\rightarrow u^0$ in $\E_\beta$.
\end{cor}

Next, we note that the solution possesses more regularity if the date are more regular.

\begin{prop}\label{P1}
Suppose that 
\bqn\label{P10}
u^0\in \E_\beta\cap C^1(\R^+,E_0)\cap C(\R^+,E_1)\ .
\eqn
In addition, let
\bqn\label{P11}
B[u]\in C^1(J,E_0)\cap C(J,E_1)\qquad\text{and}\qquad
m_{\bar{u}}\in C^{0,1}(J\times\R^+)\cup C^{1,0}(J\times\R^+)\ .
\eqn
Then, for all $t\in\dot{J}$ and $a>0$, we have
    \begin{align}
    \partial_t u(t,\cdot)\,,\, \partial_a u(t,\cdot)\in
    C([0,t],E_0)\cap C((t,\infty),E_0)\ ,\label{r1}\\
    \partial_t u(\cdot,a)\,,\, \partial_a u(\cdot,a)\in
    C([0,a)\cap J,E_0)\cap C([a,\infty)\cap J,E_0)\ ,\label{r2}
    \end{align}
and $u$ solves \eqref{1}-\eqref{3} pointwise in $E_0$ for $t\not= a$.
\end{prop}

Since $u$ represents a density in applications, one expects it to be non-negative. The next result establishes this positivity result if $E_0$ is an ordered B-space with positive cone $E_0^+$. In this case we put $$\E_\theta^+:=L_1(\R^+,E_\theta^+,\g (a)\rd a)\quad\text{with}\quad E_\theta^+:=E_\theta\cap E_0^+\ .$$ 
We refer to \cite{LQPP} for more information about operators on ordered B-spaces.

\begin{prop}\label{P2}
Suppose that $E_0$ is an ordered B-space with positive cone $E_0^+$. Given $T>0$, $\theta>0$, and $\bar{v}\in C^\theta([0,T],E_\alpha)$ let the linear operator $A[\bar{v}](t)$ be resolvent positive for each $t\in [0,T]$. Further suppose that $B$ maps $C(I_T,\E_\alpha^+)$ into $C(I_T,E_\mu^+)$. Then $u^0\in\E_\beta^+$ implies $u(t)\in \E_\beta^+$ for $t\in J$.
\end{prop}

We next focus on global existence. Due to the quasi-linear structure of equation \eqref{1} it is clearly not obvious how to derive estimates like \eqref{global} in general. The next result aims at providing conditions ensuring \eqref{global}.

\begin{prop}\label{P3}
Let $\vartheta\in [0,1]$ with $(\vartheta,\beta)\not= (0,1)$. Suppose that for each $T>0$ there are numbers $\varrho>0$, $\sigma\in\R$, $\kappa\ge 1$, $\omega>0$, and $c_1>0$ depending possibly on $T$ such that
\bqn\label{i}
\sigma+A[\bar{u}]\in C\big(J_T,\mathcal{H}(E_1,E_0;\kappa,\omega)\big)
\eqn
with
\bqn\label{ii}
\left\| A[\bar{u}](t)-A[\bar{u}](t_*)\right\|_{\mathcal{L}(E_1,E_0)}\le c_1 \vert t-t_*\vert^\varrho\ ,\quad t,t_*\in J_T\ ,
\eqn
and
\bqn\label{iii}
\left\| B[u](t)\right\|_{E_\vartheta}\le c_1\big(1+\max_{0\le \tau\le t}\| u(\tau)\|_{\E_\beta}\big) \ ,\quad t\in J_T\ ,
\eqn
where $J_T:=J\cap [0,T]$ for $T>0$. Further suppose that $m$ a non-negative or bounded. Then the solution $u$ exists globally, that is, $J=\R^+$.
\end{prop}

\begin{rem}\label{R4}
If the constant $c_0$ in \eqref{B} does not depend on $R$, then \eqref{iii} is a consequence of \eqref{B}. Also, condition \eqref{iii} may be replaced by 
\bqn\label{iiii}
\left\| B[u](t)\right\|_{E_\vartheta}\le c_2\big(1+\| u(t)\|_{\E_\upsilon}\big) \ ,\quad t\in J_T\ ,
\eqn
for $(0,1)\not=(\vartheta,\upsilon)\in [0,1]^2$ and some $c_2=c_2(T)>0$. This latter condition is slightly weaker than \eqref{iii} with respect to regularity since we may allow for $\upsilon>\beta$, but it somehow assumes $B$ to depend locally on $u$ with respect to time.
\end{rem}

For the proofs of Corollary \ref{C1}, Propositions \ref{P1}-\ref{P3}, and Remark \ref{R4} we refer to Section \ref{further proofs}.

\section{Proof of Theorem \ref{T}}\label{proof of theorem T}

\noindent Given the assumptions of Theorem \ref{T}, let $\gamma\in [\alpha,\beta)$ be arbitrary and choose $\theta\in (0,\zeta\wedge (\beta-\gamma))$. We fix any $T_0>0$ and $R$ with 
\bqn\label{5}
R>\left(1+g_1 \,\|\g \|_{L_\infty (0,T_0)}\right)\, \left\|u^0\right\|_{\E_\beta}\, \max_{0\le s\le t\le T_0}\| U_{A[0]}(t,s)\|_{\mathcal{L}(E_\beta,E_\gamma)}\ .
\eqn
For $T\in (0,T_0)$ set
\bqn\label{vv}
\mathcal{V}_T:=\big\{ u\in C(I_T,\E_\gamma)\, ;\, \| u(t)\|_{\E_\gamma}\le R+1\, ,\,\|\bar{u}(t)-\bar{u}(t_*)\|_{E_\gamma}\le \vert t-t_*\vert^\theta\,,\, 0\le t,t_*\le T\big\}\ ,
\eqn 
where
$$
\bar{v}:=\int_0^\infty v(a)\, \h(a)\ \rd a\ ,\quad v\in \E_0\ ,
$$
and observe that $\mathcal{V}_T$, equipped with the topology induced by $C(I_T,\E_\gamma)$, is a complete metric space. Also note that $(A_1)$ ensures the existence of a constant $c_1>0$ such that \bqn\label{66o}
0\le \h(a)\le c_1\g (a)\ ,\quad a\ge 0\ ,
\eqn
and hence
\bqn\label{66}
\|\bar{u}\|_{E_\vartheta}\,\le\, c_1\,\|u\|_{\E_\vartheta}\ ,\quad u\in\E_\vartheta\ ,\quad \vartheta\in [0,1]\ .
\eqn 
In particular, due to the embedding $E_\gamma\hookrightarrow E_\alpha$ there is a constant $c(R)>0$ for which
$$
\|\bar{u}\|_{C^\theta(I_T,E_\alpha)}\le c(R)\ ,\quad u\in\mathcal{V}_T\ ,
$$ 
and thus there are numbers $\rho\in (0,1)$, $\omega>0$, $\kappa\ge 1$, \mbox{$\sigma\in \R$}, and $c_0>0$ depending on $T_0$ and $R$ such that \eqref{t5}, \eqref{t6} hold for $u,u_*\in\mathcal{V}_T$. Therefore, invoking Lemma~II.5.1.3,
Lemma~II.5.1.4, and Equation~(II.5.3.8) in \cite{LQPP}, we conclude that there exists $c(T_0,R)>0$ such that unique evolution systems $U_{A[\bar{u}]}$ and $U_{A[\bar{u}_*]}$ on $E_0$ corresponding to any $u,u_*\in \mathcal{V}_T$ satisfy
    \bqn\label{y1}
    \|U_{A[\bar{u}]}(t,s)\|_{\mathcal{L}(E_\sigma)}\,+\,
    (t-s)^{\tau-\sigma}\|U_{A[\bar{u}]}(t,s)\|_{\mathcal{L}(E_\upsilon,
    E_\tau)}\, \le\, c(T_0,R)
    \eqn
for $0\le s<t\le T$ and $0\le \sigma <\upsilon\le \tau\le 1$,
     \bqn\label{y2}
    \|U_{A[\bar{u}]}(t,r)-U_{A[\bar{u}]}(s,r)\|_{\mathcal{L}(E_\tau,E_\upsilon)}\,
    \le\,  c(T_0,R)\, (t-s)^{\tau-\upsilon}
    \eqn
for $0\le r<s<t\le T$ and $0< \upsilon \le \tau < 1$, as well as
     \bqn\label{y3}
     \begin{split}
    \|U_{A[\bar{u}]}(t,s)-U_{A[\bar{u}_*]}(t,s)\|_{\mathcal{L}(E_\sigma,E_\upsilon)}\,&\le\,
    c(T_0,R)\, (t-s)^{\sigma-\upsilon}\,\|\bar{u}-\bar{u}_*\|_{C(I_T,E_\alpha)}\\
    &\le\,
    c(T_0,R)\, (t-s)^{\sigma-\upsilon}\,\|u-u_*\|_{\mathcal{V}_T}
    \end{split}
    \eqn
for $0\le s<t\le T$ and $0\le \sigma, \upsilon\le 1$ with $\sigma\not=
0$, $\upsilon\not= 1$.

Next, for $u,u_*\in\mathcal{V}_T$ we have $B[u]\in C(I_T,E_\mu)$ by $(A_4)$ with
\bqn\label{8}
\begin{split}
\| B[u]-B[u_*]\|_{C(I_T,E_\mu)}\,&\le\, c(T_0,R)\,\|u-u_*\|_{C(I_T,\E_\alpha)}\\
&\le\, c(T_0,R)\,\|u-u_*\|_{\mathcal{V}_T}\ ,
\end{split}
\eqn
whence
\bqn\label{8c}
\|B[u]\|_{C(I_T,E_\mu)}\,\le\, c(T_0,R)\ .
\eqn
Also note that, for $u,u_*\in\mathcal{V}_T$ and $t\in [0,T]$,
\bqn\label{9}
\begin{split}
\left\vert \int_0^a  m_{\bar{u}}  \right. &\left. (s+t-a,s)\rd s-\int_0^am_{\bar{u}_*}(s+t-a,s) \rd s \right\vert \\ 
&  + \left\vert    \int_0^t m_{\bar{u}}(s,s+b-t)\rd s -\int_0^t m_{\bar{u}_*}(s,s+b-t)\rd s \right\vert
\le\, c(T_0,R)\, \|u-u_*\|_{\mathcal{V}_T} 
\end{split}
\eqn
provided $0\le a\le t<b$. Defining $\Theta$ by
    \bqnn
     \Theta(u)(t,a)\, :=\, \left\{ \begin{aligned}
    &e^{-\int_0^am_{\bar{u}}(s+t-a,s)\rd s}\, U_{A[\bar{u}]}(t,t-a)\, B[u](t-a)\ , & 0\le a< t&\ ,\\
    & e^{-\int_0^t m_{\bar{u}}(s,s+a-t)\rd s}\, U_{A[\bar{u}]}(t,0)\, u^0(a-t)\ , & 0\le t<a&\ ,
    \end{aligned}
   \right.
    \eqnn 
for $0\le t\le T$, $a>0$, and $u\in\mathcal{V}_T$, we claim that $\Theta:\mathcal{V}_T\rightarrow \mathcal{V}_T$ is a contraction provided that $T=T(R)\in (0,T_0)$ is chosen sufficiently small. In the following, let $\bar{\mu}\in (0,\mu)$.

We first prove that $\Theta(u)\in C(I_T,\E_\beta)\hookrightarrow C(I_T,\E_\gamma)$ for $u\in\mathcal{V}_T$. To this end, observe that assumptions on $\g$ imply
\bqn\label{gg}
\g(a)\le g_1\, \|\g\|_{L_\infty(0,T_0)}\, \g(a-t)\ ,\quad a>t\ ,\quad 0\le t\le T\ . 
\eqn
Hence, recalling that $g\in L_{\infty,loc}(\R^+)$ and using \eqref{mm}, \eqref{y1}, \eqref{y2}, \eqref{8c}, and \eqref{gg} we estimate for $u\in\mathcal{V}_T$ and $0\le t\le t_*\le T$
\bqnn
\begin{split}
\| & \Theta(u)(t)-\Theta(u)(t_*)\|_{\E_\beta} \,\\ 
\le &\int_0^t \left\vert e^{-\int_0^am_{\bar{u}}(s+t_*-a,s)\rd s}-e^{-\int_0^am_{\bar{u}}(s+t-a,s)\rd s}\right\vert \left\| U_{A[\bar{u}]}(t_*,t_*-a)\right\|_{\mathcal{L}(E_\mu,E_\beta)}\\
&\qquad\qquad\qquad\qquad \times  \left\| B[u](t_*-a)\right\|_{E_\mu}\, \g(a)\,\rd a\\
& + \int_0^t\left\vert e^{-\int_0^am_{\bar{u}}(s+t-a,s)\rd s}\right\vert  \left\| \left[ U_{A[\bar{u}]}(t_*,t_*-a)-U_{A[\bar{u}]}(t,t-a)\right] B[u](t_*-a)\right\|_{E_\beta}\, \g(a)\,\rd a\\
&+ \int_0^t\left\vert e^{-\int_0^am_{\bar{u}}(s+t-a,s)\rd s}\right\vert  \left\|  U_{A[\bar{u}]}(t,t-a) \right\|_{\mathcal{L}(E_\mu,E_\beta)}\,  \left\|  B[u](t_*-a)- B[u](t-a)\right\|_{E_\mu}\, \g(a)\,\rd a\\
& + \int_t^{t_*}\left\vert e^{-\int_0^am_{\bar{u}}(s+t_*-a,s)\rd s}\right\vert  \left\|  U_{A[\bar{u}]}(t_*,t_*-a) \right\|_{\mathcal{L}(E_\mu,E_\beta)}\,  \left\|  B[u](t_*-a)\right\|_{E_\mu}\, \g(a)\,\rd a\\
&+ \int_t^{t^*}   \left\vert e^{-\int_0^t m_{\bar{u}}(s,s+a-t)\rd s}\right\vert\,        
  \left\| U_{A[\bar{u}]}(t,0)\right\|_{\mathcal{L}(E_\beta)} \,  \left\|u^0(a-t)\right\|_{E_\beta}\, \g(a)\,\rd a\\
&+ \int_{t^*}^\infty   \left\vert e^{-\int_0^{t_*} m_{\bar{u}}(s,s+a-t_*)\rd s}-e^{-\int_0^t m_{\bar{u}}(s,s+a-t)\rd s}\right\vert\,        
  \left\|  U_{A[\bar{u}]}(t_*,0) \right\|_{\mathcal{L}(E_\beta)}\,  \left\|  u^0(a-t_*)\right\|_{E_\beta}\, \g(a)\,\rd a\\
&+ \int_{t^*}^\infty   \left\vert e^{-\int_0^t m_{\bar{u}}(s,s+a-t)\rd s}\right\vert\,        
  \left\| \left[ U_{A[\bar{u}]}(t_*,0)-U_{A[\bar{u}]}(t,0)\right] \, u^0(a-t_*)\right\|_{E_\beta}\, \g(a)\,\rd a\\
&+ \int_{t^*}^\infty   \left\vert e^{-\int_0^t m_{\bar{u}}(s,s+a-t)\rd s}\right\vert\,        
  \left\| U_{A[\bar{u}]}(t,0)\right\|_{\mathcal{L}(E_\beta)} \, \left\| u^0(a-t_*)-u^0(a-t)\right\|_{E_\beta}\, \g(a)\,\rd a
\\
\le &\, 
c(T_0,R) \int_0^t \int_0^a \left\vert m_{\bar{u}}(s+t_*-a,s)-m_{\bar{u}}(s+t-a,s)\right\vert\rd s \, a^{\bar{\mu}-\beta}\,\rd a\\
& +c(T_0,R) \int_0^t\left\| \left[ U_{A[\bar{u}]}(t_*,t_*-a)-U_{A[\bar{u}]}(t,t-a)\right] B[u](t_*-a)\right\|_{E_\beta}\, \rd a\\
&+ c(T_0,R) \int_0^t a^{\bar{\mu}-\beta}\,  \left\|  B[u](t_*-a)- B[u](t-a)\right\|_{E_\mu}\, \rd a\\
&+ c(T_0,R)\int_t^{t^*}   a^{\bar{\mu}-\beta}\,\rd a\, +\, c(T_0,R)\int_0^{t_*-t}\|u^0(a)\|_{E_\beta}\, \g(a)\, \rd a \\
&+ c(T_0,R)\, \int_{0}^\infty \left\vert \int_0^{t_*} m_{\bar{u}}(s,s+a)\rd s-\int_0^t m_{\bar{u}}(s,s+a+t_*-t)\rd s \right\vert\, \left\|  u^0(a)\right\|_{E_\beta}\, \g(a)\,\rd a\\
&+ c(T_0,R)\int_{0}^\infty \left\| \left[ U_{A[\bar{u}]}(t_*,0)-U_{A[\bar{u}]}(t,0)\right]  u^0(a)\right\|_{E_\beta}\, \g(a)\,\rd a\\
&+ c(T_0,R) \int_0^\infty \left\| u^0(a)-u^0(a-t+t_*)\right\|_{E_\beta}\, \g(a)\,\rd a\\
=:&\, I+II+\ldots+VIII \ .
\end{split}
\eqnn
Now, as $\vert t-t_*\vert\rightarrow 0$ we clearly have $I+IV+V+VI\rightarrow 0$ due the Lebesgue Theorem (we obviously may assume $\beta\ge \mu$). Using $B[u]\in C(I_T,E_\mu)$, the density of the embedding $E_\beta\hookrightarrow E_\mu$, and the fact that the evolution system $U_{A[\bar{u}]}$ is uniformly strongly continuous on compact subsets of $E_\beta$, we also derive that $II\rightarrow 0$. The continuity of $B[u]$ also entails $III\rightarrow 0$, while the strong continuity of $U_{A[\bar{u}]}$ on $E_\beta$ ensures $VII\rightarrow 0$. Finally, $VIII\rightarrow 0$ holds since translations are strongly continuous. Therefore, $\Theta(u)\in C(I_T,\E_\beta)$.

Next observe that \eqref{y3} implies
\bqn\label{100}
\left\| U_{A[\bar{u}]}(t,s)\right\|_{\mathcal{L}(E_\beta,E_\gamma)}\,\le\, c(T_0,R)(t-s)^{\beta-\gamma}\,+\, c_2\ ,\quad 0\le s<t\le T\ ,\quad u\in\mathcal{V}_T\ ,
\eqn
where 
$$
c_2:=\max_{0\le s\le t\le T_0}\,\left\| U_{A[0]}(t,s)\right\|_{\mathcal{L}(E_\beta, E_\gamma)}\ .
$$
In view of \eqref{5}, \eqref{y1}, \eqref{8c}, \eqref{gg}, \eqref{100}, and assumption $(A_4)$ we deduce, for $u\in\mathcal{V}_T$ and $t\in I_T$, that
\bqnn
\begin{split}
\left\|\Theta(u)(t)\right\|_{\E_\gamma}\,&\le\, c(T_0,R)\int_0^t \left\| U_{A[\bar{u}]}(t,t-a)\right\|_{\mathcal{L}(E_\mu, E_\gamma)}\, \left\| B[u](t-a)\right\|_{E_\mu}\, \g(a)\ \rd a\\
&\quad +c(T_0,R)\int_t^\infty \left\| U_{A[\bar{u}]}(t,0)\right\|_{\mathcal{L}(E_\beta, E_\gamma)}\, \left\| u^0(a-t)\right\|_{E_\beta}\, \g(a)\ \rd a\\
&\le \, c(T_0,R)\, t^{1+\bar{\mu}-\gamma}\,+\, c(T_0,R)\, t^{\beta-\gamma}\,\left\| u^0\right\|_{\E_\beta}\,+\, c_2 \, g_1\,\| \g\|_{L_\infty(0,T_0)}\,\left\| u^0\right\|_{\E_\beta}\ .
\end{split}
\eqnn
Since $\gamma<\beta$ we may choose $T=T(R)\in (0,T_0)$ sufficiently small to obtain
\bqn\label{667}
\left\|\Theta(u)(t)\right\|_{\E_\gamma}\,\le\, R+1\ ,\quad t\in I_T\ ,\quad u\in \mathcal{V}_T\ .
\eqn
Moreover, writing for $u\in \mathcal{V}_T$ and $t\in I_T$
\bqn\label{f}
\begin{split}
\overline{\Theta(u)}(t)=\int_0^\infty \Theta(u)(t,a)\, \h(a)\ \rd a\,& =\int_0^t 
e^{-\int_0^{t-a} m_{\bar{u}}(a+s,s)\rd s}\,  U_{A[\bar{u}]}(t,a)\, B[u](a)\, \h(t-a)\ \rd a\\
&\quad +\int_0^\infty e^{-\int_0^{t} m_{\bar{u}}(s,a+s)\rd s}\, U_{A[\bar{u}]}(t,0)\, u^0(a)\, \h(a+t)\ \rd a
\end{split}
\eqn
and using the fact that, for $0\le a\le t\le t_*\le T$,
\bqnn
\begin{split}
\left\| U_{A[\bar{u}]}(t_*,a)-U_{A[\bar{u}]}(t,a)\right\|_{\mathcal{L}(E_\mu,E_\gamma)}\,&\le\,\left\| U_{A[\bar{u}]}(t_*,t)-U_{A[\bar{u}]}(t,t)\right\|_{\mathcal{L}(E_\beta,E_\gamma)} \, \left\| U_{A[\bar{u}]}(t,a)\right\|_{\mathcal{L}(E_\mu,E_\beta)}\\
&\le\, c(T_0,R)\, \vert t_*-t\vert^{\beta-\gamma}\, (t-a)^{\bar{\mu}-\beta}
\end{split}
\eqnn
by \eqref{y1} and \eqref{y2}, we derive from $(A_1)$, $(A_4)$, \eqref{5}, \eqref{66o}, \eqref{y1}, \eqref{y2}, and \eqref{8c} that, for $0\le t\le t_*\le T$,
\bqnn
\begin{split}
\big\|  &\overline{\Theta(u)}(t)-\overline{\Theta(u)}(t_*)\big\|_{\E_\gamma}\\
\,\le &\int_0^t \left\vert e^{-\int_0^{t_*-a} m_{\bar{u}}(a+s,s)\rd s}-e^{-\int_0^{t-a} m_{\bar{u}}(a+s,s)\rd s}\right\vert\,  \left\| U_{A[\bar{u}]}(t_*,a)\right\|_{\mathcal{L}(E_\mu,E_\gamma)}\,\left\| B[u](a)\right\|_{E_\mu}\, \h(t-a)\ \rd a\\
&\quad + c(T_0,R) \int_0^t \left\|  U_{A[\bar{u}]}(t_*,a)-U_{A[\bar{u}]}(t,a)\right\|_{\mathcal{L}(E_\mu,E_\gamma)}\,\left\| B[u](a)\right\|_{E_\mu}\, \h(t-a)\ \rd a\\
&\quad + c(T_0,R) \int_0^t \left\|  U_{A[\bar{u}]}(t,a)\right\|_{\mathcal{L}(E_\mu,E_\gamma)}\,\left\| B[u](a)\right\|_{E_\mu}\, \left\vert\h(t_*-a)-\h(t-a)\right\vert\ \rd a\\
&\quad + c(T_0,R) \int_t^{t_*} \left\|  U_{A[\bar{u}]}(t_*,a)\right\|_{\mathcal{L}(E_\mu,E_\gamma)}\,\left\| B[u](a)\right\|_{E_\mu}\, \left\vert\h(t-a)\right\vert\ \rd a\\
&\quad +\int_0^\infty \left\vert e^{-\int_0^{t_*} m_{\bar{u}}(s,a+s)\rd s}-e^{-\int_0^{t} m_{\bar{u}}(s,a+s)\rd s}\right\vert \, \left\| U_{A[\bar{u}]}(t_*,0)\right\|_{\mathcal{L}(E_\beta,E_\gamma)}\, \left\| u^0(a)\right\|_{E_\beta}\, \h(a+t_*)\ \rd a\\
&\quad +c(T_0,R)\int_0^\infty  \left\| U_{A[\bar{u}]}(t_*,0)-U_{A[\bar{u}]}(t,0)\right\|_{\mathcal{L}(E_\beta,E_\gamma)}\, \left\| u^0(a)\right\|_{E_\beta}\, \h(a+t_*)\ \rd a\\
&\quad +c(T_0,R)\int_0^\infty  \left\| U_{A[\bar{u}]}(t,0)\right\|_{\mathcal{L}(E_\beta,E_\gamma)}\, \left\| u^0(a)\right\|_{E_\beta}\, \vert \h(a+t_*)-\h(a+t)\vert\ \rd a\\
\le & c(T_0,R)\left\{\vert t_*-t\vert +\vert t_*-t\vert^{\beta-\gamma} +\int_0^t (t-a)^{\bar{\mu}-\gamma}\vert\h(t_*-a)-\h(t-a) \vert\rd a\right. \\
&\qquad\qquad\qquad\qquad +\vert t_*-t\vert^{1+\bar{\mu}-\gamma} +\vert t_*-t\vert + \vert t_*-t\vert^{\beta-\gamma}+ \vert t_*-t\vert^\zeta\bigg\}\ .
\end{split}
\eqnn
Taking into account that, due to $(A_1)$,
$$
\int_0^t (t-a)^{\bar{\mu}-\gamma}\vert\h(t_*-a)-\h(t-a) \vert\, \rd a =\int_0^t a^{\bar{\mu}-\gamma}\vert\h(t_*-t+a)-\h(a) \vert\, \rd a\le c(T_0)\,\vert t_*-t\vert^\zeta
$$
and recalling the choice of $\theta$, we may make $T=T(R)\in (0,T_0)$ smaller, if necessary, and conclude that
\bqn\label{h}
\big\|\overline{\Theta(u)}(t)-\overline{\Theta(u)}(t_*)\big\|_{\E_\gamma}\le \vert t_*-t\vert^\theta\ ,\quad 0\le t\le t_*\le T\ .
\eqn

To prove that $\Theta$ is contractive, we observe that assumption $(A_4)$ together with \eqref{5}, \eqref{y1}, \eqref{y3}, \eqref{8}, \eqref{8c}, and \eqref{9} imply that, for $u,u_*\in\mathcal{V}_T$, $0\le t\le T\le T_0$, and for all $\xi\in [0,\beta]$,
\bqnn
\begin{split}
\big\|  \Theta(u)(t)&-\Theta(u_*)(t)\big\|_{\E_\xi}\\
\,\le &\, c(T_0,R) \int_0^t \int_0^a \left\vert m_{\bar{u}}(t-a+s,s)- m_{\bar{u}_*}(t-a+s,s)\right\vert\,\rd s\,  \left\| U_{A[\bar{u}]}(t,t-a)\right\|_{\mathcal{L}(E_\mu,E_\xi)}\\
&\qquad\qquad\qquad\qquad\qquad\times\left\| B[u](t-a)\right\|_{E_\mu}\, \g(a)\ \rd a\\
&\quad + c(T_0,R) \int_0^t \left\|  U_{A[\bar{u}]}(t,t-a)-U_{A[\bar{u}_*]}(t,t-a)\right\|_{\mathcal{L}(E_\mu,E_\xi)}\,\left\| B[u](t-a)\right\|_{E_\mu}\, \g(a)\ \rd a\\
&\quad + c(T_0,R) \int_0^t \left\|  U_{A[\bar{u}]}(t,t-a)\right\|_{\mathcal{L}(E_\mu,E_\xi)}\,\left\| B[u](t-a)-B[u_*](t-a)\right\|_{E_\mu}\, \g(a)\ \rd a
\end{split}
\eqnn
\bqnn
\begin{split}
&\quad +c(T_0,R)\int_t^\infty  \int_0^t\left\vert  m_{\bar{u}}(s,a-t+s)\rd s- m_{{\bar{u}}_*}(s,a-t+s)\rd s\right\vert \, \left\| U_{A[\bar{u}]}(t,0)\right\|_{\mathcal{L}(E_\beta,E_\xi)}\\
&\qquad\qquad\qquad\qquad\qquad\times \left\| u^0(a-t)\right\|_{E_\beta}\, \g(a)\ \rd a\\
&\quad +c(T_0,R)\int_t^\infty  \left\| U_{A[\bar{u}]}(t,0)-U_{A[\bar{u}_*]}(t,0)\right\|_{\mathcal{L}(E_\beta,E_\xi)}\, \left\| u^0(a-t)\right\|_{E_\beta}\, \g(a)\ \rd a\\
\le&\, c(T_0,R)\,\left\|u-u_*\right\|_{\mathcal{V}_T}\,\big\{t^{1+\bar{\mu}-\xi}+t^{1+\mu-\xi}+
t^{1+\bar{\mu}-\xi}+t+t^{\beta-\xi}\big\}\ ,
\end{split}
\eqnn
that is 
\bqn\label{theta}
\big\|  \Theta(u)(t)-\Theta(u_*)(t)\big\|_{\E_\xi}\,\le\,c(T_0,R)\,\big(t^{\bar{\mu}}+t^{\beta-\xi}\big)\, \left\|u-u_*\right\|_{\mathcal{V}_T}\ ,\quad t\in I_T\ .
\eqn
In particular, taking $\xi=\gamma<\beta$ we may choose $T=T(R)\in (0,T_0)$ sufficiently small such that
$$
\big\|  \Theta(u)-\Theta(u_*)\big\|_{\mathcal{V}_T} \le \,\frac{1}{2}\, \left\|u-u_*\right\|_{\mathcal{V}_T}\ ,\quad u,u_*\in\mathcal{V}_T\ .
$$
Therefore, $\Theta: \mathcal{V}_T\rightarrow \mathcal{V}_T$ is a contraction provided $T=T(R)\in (0,T_0)$ is sufficiently small and hence possesses a unique fixed point, say $u$, in $\mathcal{V}_T\cap C(I_T,\E_\beta)$. Consequently,
    \bqn\label{uu}
     u(t,a)\, 
     =\, \left\{ \begin{aligned}
    &e^{-\int_0^am_{\bar{u}}(s+t-a,s)\rd s}\, U_{A[\bar{u}]}(t,t-a)\, B[u](t-a)\ , & 0\le a< t&\ ,\\
    & e^{-\int_0^t m_{\bar{u}}(s,s+a-t)\rd s}\, U_{A[\bar{u}]}(t,0)\, u^0(a-t)\ , & 0\le t<a&\ ,
    \end{aligned}
   \right.
    \eqn
for $0\le t\le T$ and $a>0$. An estimate similar to \eqref{667} combined with the strong continuity properties of the evolution system $U_{A[\bar{u}]}$ then warrants that
\bqn\label{668}
u\in C_{\upsilon-\beta}\big( (0,T],\E_\upsilon\big)\ ,\quad \beta \le \upsilon\le 1\ .
\eqn
In order to extend the just found solution $u\in C(I_T,\E_\beta)$, we choose
\bqn\label{55}
R>\left(1+g_1 \,\|\g \|_{L_\infty (0,T_0)}\right)\, \|u\|_{C(I_T,\E_\beta)}\, \max_{0\le s\le t\le T_0}\| U_{A[0]}(t,s)\|_{\mathcal{L}(E_\beta,E_\gamma)}\ .
\eqn
similarly as in \eqref{5}, and take now $\mathcal{V}_S$ to be
$$
\mathcal{V}_S:=\big\{ v\in C(I_S,\E_\gamma)\, ;\, \| v(t)\|_{\E_\gamma}\le R+1\, ,\,\|\bar{v}(t)-\bar{v}(t_*)\|_{E_\gamma}\le \vert t-t_*\vert^\theta\,,\, 0\le t,t_*\le S\,,\, v(0)=u(T)\big\}\ ,
$$
for $S>0$ with $T+S\le T_0$. Given $v\in\mathcal{V}_S$, we put
$$
V(t):=\left\{\begin{array}{ll} u(t)\ , & 0\le t\le T\ ,\\ v(t-T)\ , & T\le t\le T+S\ ,
\end{array}\right.
$$
and obtain $V\in C(I_{T+S},\E_\gamma)$ with $\|\bar{V}\|_{C^\theta(I_{T+S},E_\gamma)}\le R+2$. We then introduce
$$
\hat{A}[\bar{v}](t):=A[\bar{V}](t+T)\ ,\quad \hat{B}[v](t):=B[V](t+T)\ ,\quad \hat{m}_{\bar{v}}(t,a):=m(t+T,a,\bar{v}(t))
$$
for $t\in I_S$ and $a>0$. It follows from assumption $(A_2)$ that
$$
\sigma+\hat{A}[\bar{v}]\in C\big(I_S,\mathcal{H}(E_1,E_0;\kappa,\omega)\big)
$$ 
and 
$$
\hat{A}[\bar{v}]\in C^\rho\big(I_S,\mathcal{L}(E_1,E_0)\big)\quad \text{with}\quad  \|\hat{A}[\bar{v}]\|_{ C^\rho(I_S,\mathcal{L}(E_1,E_0))}\le c(T_0,R)
$$
for some $\sigma$, $\omega$, $\kappa$, $\rho$ depending on $R$ and $T_0$.
If also $v_*\in\mathcal{V}_S$, then
$$
\big\| \hat{A}[\bar{v}]-\hat{A}[\bar{v}_*]\big\|_{C(I_S,\mathcal{L}(E_1,E_0))}\le c(T_0,R)\, \| \bar{v}-\bar{v}_*\|_{C(I_S,E_\alpha)}\ .
$$
Hence $\hat{A}$ satisfies \eqref{t5} and \eqref{t6}. Moreover, for $v, v_*\in\mathcal{V}_S$ we also have
$$
\big\| \hat{B}[v]-\hat{B}[v_*]\big\|_{C(I_S,E_\mu)}\le c(T_0,R)\, \| v-v_*\|_{C(I_S,\E_\alpha)}
$$
by $(A_3)$, that is, $\hat{B}$ satisfies \eqref{B}. Taking $S=S(R)>0$ sufficiently small we deduce as before the existence of a function $v\in C(I_S,\E_\beta)$ with
 \bqn\label{v}
     v(t,a)\, 
     =\, \left\{ \begin{aligned}
    &e^{-\int_0^a\hat{m}_{\bar{v}}(s+t-a,s)\rd s}\, U_{\hat{A}[\bar{v}]}(t,t-a)\, \hat{B}[v](t-a)\ , & 0\le a< t&\ ,\\
    & e^{-\int_0^t \hat{m}_{\bar{v}}(s,s+a-t)\rd s}\, U_{\hat{A}[\bar{v}]}(t,0)\, u(T,a-t)\ , & 0\le t<a&\ ,
    \end{aligned}
   \right.
    \eqn
for $0\le t\le S$ and $a>0$. We then extend the function $u$ by $w:I_{T+S}\rightarrow\E_\beta$ being defined as
$$
w(t):=\left\{\begin{array}{ll} u(t)\ , & 0\le t\le T\ ,\\ v(t-T)\ , & T\le t\le T+S\ .
\end{array}\right.
$$
Clearly, owing to $w\big\vert_{I_T}=u$ we infer $B[w]\big\vert_{I_T}=B[u]$ and $A[\bar{w}]\big\vert_{I_T}=A[\bar{u}]$ from assumptions $(A_2)$, $(A_3)$. Consequently, $U_{A[\bar{w}]}(t,s)=U_{A[\bar{u}]}(t,s)$ for $0\le s\le t\le T$. Hence the function $w$ still satisfies \eqref{uu} in which $u$ is replaced by $w$ everywhere. Next, since
$$
\hat{A}[\bar{v}](t-T)=A[\bar{w}](t)\ ,\quad T\le t\le T+S\ ,
$$
we have by uniqueness
$$
U_{\hat{A}[\bar{v}]}(t-T,s)=U_{A[\bar{w}]}(t,s+T)\ ,\quad 0\le s\le t-T\le S\ .
$$
Furthermore,
$$
\hat{B}[v](t-T-a)=B[w](t-a)\ ,\quad T\le t\le T+S\ ,\quad 0\le t-T-a\ .
$$
From these observations and using \eqref{uu} and \eqref{v} it is then straightforward that
 \bqnn
     w(t,a)=v(t-T,a)\, 
     =\, \left\{ \begin{aligned}
    &e^{-\int_0^am_{\bar{w}}(s+t-a,s)\rd s}\, U_{A[\bar{w}]}(t,t-a)\, B[w](t-a)\ , & 0\le a< t&\ ,\\
    & e^{-\int_0^t m_{\bar{w}}(s,s+a-t)\rd s}\, U_{A[\bar{w}]}(t,0)\, u^0(a-t)\ , & 0\le t<a&\ ,
    \end{aligned}
   \right.
    \eqnn
for $T\le t\le T+S$. Therefore, we may extend $u$ to a unique maximal generalized solution $u$ in $C(J,\E_\beta)$ satisfying \eqref{uu} for $t\in J$ and $a>0$. Clearly, the maximal interval of existence, $J$, is open in $[0,\infty)$. If \eqref{global} holds true, then \eqref{55} and the above extension procedure yield $J=\R^+$. Obviously, \eqref{668} holds for any $T\in \dot{J}$. Moreover, proceeding as in \eqref{h} shows that for any function $\tilde{h}$ satisfying $(A_1)$ we have
$$
\int_0^\infty u(\cdot,a)\,\tilde{\h}(a)\, \rd a \in C^{\zeta\wedge (\gamma-\beta)}(J,E_\gamma)
$$ 
for $\gamma\in [\alpha,\beta)$. Taking $\h\equiv 1$, \eqref{f} reads
\bqnn
\begin{split}
\int_0^\infty u(t,a)\ \rd a\, &=\int_0^t U_{A[\bar{u}]}(t,a)\,
e^{-\int_0^{t-a} m_{\bar{u}}(a+s,s)\rd s}\,  B[u](a)\ \rd a\\
&\quad+U_{A[\bar{u}]}(t,0)\, \int_0^\infty e^{-\int_0^{t} m_{\bar{u}}(s,a+s)\rd s}\, u^0(a)\ \rd a
\ .
\end{split}
\eqnn
The right hand side is clearly differentiable with respect to $t$ and, owing to \eqref{uu} and  \eqref{668} with $\upsilon=1$, we derive
$$\dfrac{\rd }{\rd t}\int_0^\infty u(t,a)\, \rd a= -A[\bar{u}](t)\int_0^\infty u(t,a)\, \rd a+B[u](t)-\int_0^\infty m_{\bar{u}}(t,a)\, u(t,a)\, \rd a
$$
from which we conclude \eqref{dd} by invoking \cite[II.Thm.1.2.2]{LQPP}. This proves Theorem \ref{T}

\section{Proof of Further Properties}\label{further proofs}

\noindent For the remainder of this section, we fix $u^0\in\E_\beta$ and let $u=u(\cdot;u^0)\in C(J,\E_\beta)\cap C_{\upsilon-\beta}((0,T],\E_\upsilon)$ for $T\in \dot{J}$ and $\beta\le \upsilon\le 1$ denote the unique maximal generalized solution to \eqref{1}-\eqref{4} on $J=J(u^0)$ corresponding to $u^0$. 

\subsection{Proof of Corollary \ref{C1}}\label{proof cor C1}

We use the notation as in the proof of Theorem \ref{T}. Given $u^0\in\E_\beta$ it is clear that we may choose $\delta>0$ such that \eqref{5} holds true if $u^0$ therein is replaced by any $u_*^0\in \E_\beta$ with $ \|u^0-u_*^0 \|_{\E_\beta}\le \delta$. Therefore, the proof of Theorem \ref{T} shows that there are solutions $u=u(\cdot;u^0)$ and $u_*=u(\cdot;u_*^0)$ both belonging to $\mathcal{V}_T$, where $T=T(R)\in (0,T_0)$ is sufficiently small. As in \eqref{theta} we obtain for any $\bar{\mu}\in (0,\mu)$, $\xi\in [0,\beta]$, and $t\in [0,T]$ 
\bqnn
\begin{split}
\| u(t)-u_*(t)\|_{\E_\xi}\,& \le\, c(T_0,R)\,\left(t^{\bar{\mu}}+t^{\beta-\xi}\right)\,\|u-u_*\|_{\mathcal{V}_T}\\
&\qquad +c(T_0,R)\int_t^\infty         
  \left\| U_{A[\bar{u}_*]}(t,0)\right\|_{\mathcal{L}(E_\beta,E_\xi)} \left\| u^0(a-t)-u_*^0(a-t)\right\|_{E_\beta} \g(a)\,\rd a\\
&\le\, c(T_0,R)\,\left(t^{\bar{\mu}}+t^{\beta-\xi}\right)\,\|u-u_*\|_{\mathcal{V}_T}\,+\, c(T_0,R)\,\left\|u^0-u_*^0\right\|_{\E_\beta}\ .
\end{split}
\eqnn
Hence, by taking $\xi=\gamma<\beta$ and making $T=T(R)\in (0,T_0)$ smaller if necessary, we first deduce
$$ 
\| u-u_*\|_{\mathcal{V}_T}\,\le\, c(T_0,R)\ \left\| u^0-u_*^0\right\|_{\E_\beta}
$$
and then, choosing $\xi=\beta$,
$$
\| u(t)-u_*(t)\|_{\E_\beta}\,\le\, c(T_0,R)\, \left\| u^0-u_*^0\right\|_{\E_\beta}\ ,\quad t\in [0,T]\ ,
$$
whence the claim of Corollary \ref{C1}.

\subsection{Proof of Proposition \ref{P1}}\label{proof prop P1}

To establish Proposition \ref{P1} we use the properties of evolution systems \cite{LQPP}
$$
\frac{\partial}{\partial t}U_{A[\bar{u}]}(t,s)w\, =\, -A[\bar{u}](t) U_{A[\bar{u}]}(t,s)w\ ,\quad 0\le s< t\in J\ ,\quad w\in E_0\ ,
$$
and
$$
\frac{\partial}{\partial s}U_{A[\bar{u}]}(t,s)v\, =\,  U_{A[\bar{u}]}(t,s)A[\bar{u}](s) v\ ,\quad 0\le s< t\in J\ ,\quad v\in E_1\ .
$$
Then, due to \eqref{P10} and \eqref{P11}, it follows from \eqref{uu} that, for $t\in\dot{J}$ and $a>0$ with $a\not=t$,
 \bqnn
    \begin{split}
    \partial_t u(t,a)\,=\,& {\bf 1}_{[a<t]}(t,a)\left\{e^{-\int_{t-a}^t  m_{\bar{u}}(s,s+a-t)\, \rd s}  U_{A[\bar{u}]}(t,t-a)\big(A[\bar{u}](t-a)+\partial_t\big)\, B[u](t-a)\right. \\
    &\quad +  \left(-m_{\bar{u}}(t,a)+ m_{\bar{u}}(t-a,0)+ \int_{t-a}^t \partial_2 m_{\bar{u}}(s,s+a-t)\, \rd a- A[\bar{u}](t)\right) u(t,a)\bigg\}\\
   &+ {\bf 1}_{[a>t]}(t,a)\left\{ \left(-m_{\bar{u}}(t,a)+ \int_{0}^t \partial_2 m_{\bar{u}}(s,s+a-t)\, \rd a- A[\bar{u}](t)\right)\, u(t,a)\right. \\
    &\qquad - e^{-\int_{0}^t  m_{\bar{u}}(s,s+a-t)\, \rd s}  U_{A[\bar{u}]}(t,0)\partial_a u^0(a-t)\bigg\} 
    \end{split}
    \eqnn
if $m\in C^{0,1}(J\times\R^+)$ and similarly    
 \bqnn
    \begin{split}
    \partial_a u(t,a)\,=\,& {\bf 1}_{[a<t]}(t,a)\left\{ \left(- m_{\bar{u}}(t-a,0)- \int_{t-a}^t \partial_2 m_{\bar{u}}(s,s+a-t)\, \rd a\right)\, u(t,a)\right. \\
    &\qquad + e^{-\int_{t-a}^t  m_{\bar{u}}(s,s+a-t)\, \rd s}  U_{A[\bar{u}]}(t,t-a)\big(-A[\bar{u}](t-a)-\partial_t\big)\, B[u](t-a)\bigg\}\\
   &+ {\bf 1}_{[a>t]}(t,a)\left\{ \left( \int_{0}^t \partial_2 m_{\bar{u}}(s,s+a-t)\, \rd a- A[\bar{u}](t)\right)\, u(t,a)\right. \\
    &\qquad + e^{-\int_{0}^t  m_{\bar{u}}(s,s+a-t)\, \rd s}  U_{A[\bar{u}]}(t,0)\partial_a u^0(a-t)\bigg\}\ .
    \end{split}
    \eqnn
Taking into account the particular form of $u$ in \eqref{uu}, the assumptions on $B$ and $u^0$, and the continuity properties of $U_{A[\bar{u}]}$ we deduce that $u$ indeed possesses the regularity \eqref{r1}, \eqref{r2} and satisfies
$$
\big(\partial_t +\partial_a\big) u(t,a) =-A[\bar{u}](t) u(t,a)-m_{\bar{u}}(t,a)u(t,a)\quad\text{in}\quad E_0$$
for $t\in \dot{J}$ and $a>0$ with $t\not=a$. Clearly,
$$ u(t,0)=B[u](t)\ ,\quad t\in \dot{J}\ , \qquad u(0,a)=u^0(a)\ ,\quad a>0\ ,
$$
where both equations hold in $E_0$. This proves Proposition \ref{P1}.

\subsection{Proof of Proposition \ref{P2}}\label{proof prop P2}

Suppose the assumptions of Proposition \ref{P2}. Replacing $\mathcal{V}_T$ in the proof of Theorem \ref{T} by the closed metric space
\bqnn
\mathcal{V}_T^+:=\big\{ v\in C(I_T,\E_\gamma^+)\, ;\, \| v(t)\|_{\E_\gamma}\le R+1\, ,\,\|\bar{v}(t)-\bar{v}(t_*)\|_{E_\gamma}\le \vert t-t_*\vert^\theta\,,\, 0\le t,t_*\le T\big\}\ ,
\eqnn 
and using the fact that the resolvent positivity of the operator $A$ implies  
$$
U_{A[\bar{v}]}(t,s): E_0^+\rightarrow E_0^+\ ,\quad 0\le s\le t\le T\ ,\quad v\in \mathcal{V}_T^+\ ,
$$
by \cite[II.§6.4]{LQPP}, it follows from the assumptions on $B$ that the map $\Theta$, introduced in the proof of Theorem~\ref{T}, is a contraction from $\mathcal{V}_T^+$ into itself (provided $T$ is chosen sufficiently small). This then readily gives Proposition \ref{P2}.

\subsection{Proof of Proposition \ref{P3}}\label{proof prop P3}

Suppose the assumptions of Proposition \ref{P3} and let $T>0$ be arbitrary. Observe that we may assume without loss of generality in \eqref{iii} that $\vartheta\le \beta$. Then \eqref{i}, \eqref{ii} in combination with \cite[II.Lem.5.1.3]{LQPP} ensure the existence of  a constant $c_3=c_3(T)>0$ such that
\bqn\label{m1}
\left\| U_{A[\bar{u}]}(t,s)\right\|_{\mathcal{L}(E_\beta)}\,+\, (t-s)^{\beta- \vartheta/2}\, \left\| U_{A[\bar{u}]}(t,s)\right\|_{\mathcal{L}(E_\vartheta,E_\beta)}\,\le\, c_3\ ,\quad 0\le s<t\in J_T\ .
\eqn
Introducing $z\in C(J_T)$ by
$$
z(\tau):=\max_{0\le t\le\tau}\| u(t)\|_{\E_\beta}\ ,\quad \tau\in J_T\ ,
$$
it follows from \eqref{iii}, \eqref{m1}, and the assumption that $m$ is non-negative or bounded
\bqn
\begin{split}
\| u(t)\|_{\E_\beta} \, &\le c(T) \int_0^t \left\| U_{A[\bar{u}]}(t,t-a)\right\|_{\mathcal{L}(E_\vartheta,E_\beta)}\,\left\| B[u](t-a)\right\|_{E_\vartheta}\,\g (a)\ \rd a\\
&\qquad +c(T) \int_t^\infty  \left\| U_{A[\bar{u}]}(t,0)\right\|_{\mathcal{L}(E_\beta)}\,\left\| u^0(a-t)\right\|_{E_\beta}\,\g (a)\ \rd a\notag\\
&\le c(T) c_1 c_3 \|\g \|_{L_\infty (J_T)}\int_0^t a^{-\beta+\vartheta/2}\,\big(1+z(t-a)\big)\ \rd a\,+\, c(T) g_1 c_3 \|\g \|_{L_\infty (J_T)}\,\left\| u^0\right\|_{\E_\beta} \label{m3}
\end{split}
\eqn
for $t\in J_T$ and thus, since $z$ is non-decreasing,
$$
z(\tau)\,\le\, c(T)\left( 1+\int_0^\tau (\tau-a)^{-\beta+\vartheta/2}\, z(a)\ \rd a\right)\ ,\quad \tau\in J_T\ .
$$
Due to $\beta-\vartheta/2<1$, Gronwall's inequality implies \eqref{global}, hence $J=\R^+$.\\

\subsection{Proof of Remark \ref{R4}}\label{proof rem r4} 
We note that if \eqref{iii} is replaced with \eqref{iiii}, we clearly may assume that $\upsilon\in [\beta,1)$. Then, as in the proof of Proposition \ref{P3} we obtain from \eqref{iiii} and the analogue of \eqref{m1}
\bqnn
\begin{split}
\| u(t)\|_{\E_\upsilon} \, 
&\le c(T) c_2 c_3 \|\g \|_{L_\infty (J_T)}\int_0^t a^{-\upsilon+\vartheta/2}\,\big(1+\|u(t-a)\|_{\E_\upsilon}\big)\ \rd a\\
&\qquad +c(T) g_1 c_2 \|\g \|_{L_\infty (J_T)}\,\left\| u^0\right\|_{\E_\beta} \, t^{-\upsilon+\beta/2}\\
&\le c(T)\left(1+t^{-\upsilon+\beta/2}+\int_0^t (t-a)^{-\upsilon+\vartheta/2}\| u(a)\|_{\E_\upsilon}\ \rd a\right)
\end{split}
\eqnn
for $t\in \dot{J}_T$. Applying the singular Gronwall inequality \cite[II.Cor.3.3.2]{LQPP}, we deduce \eqref{global}.

\section{Applications}\label{examples}

We give examples of problems to which the results of Section \ref{main result} may be applied. First we provide some conditions intended to simplify the verification of assumption $(A_2)$ and \eqref{i}, \eqref{ii}.

\subsection{General Remarks}

We show that if $A$ has a particular form, then assumption $(A_2)$ is rather easy to verify in concrete applications. This result, in particular, applies to the case when $A$ depends locally with respect to time on $\bar{u}$.

More precisely, we assume that $A$ is of the form
\bqn\label{local}
A[\bar{z}](t)=A_0\big(t,\Phi(\bar{z})(t)\big)\ ,
\eqn
where 
\bqn\label{A0}
A_0\in C_b^{\varrho,1-}\big(\R^+\times F_0, \mathcal{H}(E_1,E_0)\big)
\eqn
for some Banach space $F_0$ and $\varrho\in (0,1)$. That is, given any $R>0$ there exists $c(R)>0$ such that
$$
\|A_0(t,z)- A_0(t_*,z_*)\|_{\mathcal{H}(E_1,E_0)}\,\le\, c(R)\,\big(\vert t-t_*\vert^\varrho+\|z-z_*\|_{F_0}\big)
$$
for $t,t_*\in [0,R]$ and $z,z_*\in F_0$ with $\|z\|_{F_0}, \| z_*\|_{F_0}\le R$.

Given another Banach space $F_1$ with $F_1\hookrightarrow F_0$, the function $\Phi$ is supposed to satisfy the following conditions (for some $\alpha\in [0,1)$):
\begin{itemize}
\item[$(A_5)$] Given $T_0, R>0$ and $\theta\in (0,1)$ there are numbers $\rho\in (0,1)$ and \mbox{$c_4>0$} (depending on $T_0$, $R$, and $\theta$) such that, for each $T\in (0,T_0)$, the function $\Phi$ maps $C^\theta(I_T,E_\alpha)$ into $C^\rho(I_T,F_1)$
and satisfies
\bqnn
\left\| \Phi(\bar{z})(t)-\Phi(\bar{z})(t_*)\right\|_{F_1}\,\le c_4\, \vert t-t_*\vert^\rho
\eqnn
and
\bqnn
\left\| \Phi(\bar{z})(t)-\Phi(\bar{z}_*)(t)\right\|_{F_1}\,\le\,c_4\, \|\bar{z}-\bar{z}_*\|_{C(I_T,E_\alpha)}
\eqnn
for $t,t_*\in I_T$ and all $\bar{z},\bar{z}_*\in C^\theta(I_T,E_\alpha)$ with $\|\bar{z}\|_{C^\theta(I_T,E_\alpha)}\le R$ and $\|\bar{z}_*\|_{C^\theta(I_T,E_\alpha)}\le R$.
Moreover, if $0<T<S$ and $\bar{z},\bar{z}_*\in C(I_S,E_{\alpha})$ with $\bar{z}\big\vert_{I_T}=\bar{z}_*\big\vert_{I_T}$, then $\Phi(\bar{z})\big\vert_{I_T}=\Phi(\bar{z}_*)\big\vert_{I_T}$. 
\end{itemize}

Then we have:

\begin{prop}\label{ex1}
Suppose that the embedding $F_1\hookrightarrow F_0$ is compact and let the operator $A$ be of the form \eqref{local} with $A_0$ satisfying \eqref{A0} and $\Phi$ satisfying assumption $(A_5)$. Then $A$ satisfies assumption $(A_2)$. 
\end{prop}

\begin{proof}
Given $ T_0, R>0$ and $\theta\in (0,1)$ it follows from assumption $(A_5)$ that there exists a bounded set $M\subset F_1$ such that
$\Phi(\bar{z})(t)\in M$ for all $0\le t\le T\le T_0$ and $\bar{z}\in C^\theta(I_T,E_\alpha)$ with $\|\bar{z}\|_{C^\theta(I_T,E_\alpha)}\le R$. Due to the compactness of the embedding $F_1\hookrightarrow F_0$ we deduce that $M$ is relatively compact in $F_0$ and so is $A_0\big([0,T_0]\times M\big)$ in $\mathcal{H}(E_1,E_0)$ by continuity. Hence, \cite[I.Cor.1.3.2]{LQPP} ensures the existence of numbers $\kappa\ge1$ and $\omega>0$ such that $A_0\big([0,T_0]\times M\big)\subset \mathcal{H}(E_1,E_0;\kappa,\omega)$. But then $(A_5)$, \eqref{local}, \eqref{A0}, and the continuous embedding $F_1\hookrightarrow F_0$ readily imply $(A_2)$.
\end{proof}

It is worthwhile to point out that assumption $(A_5)$ is trivially satisfied if $\Phi$ is the identity. Therefore, assumption $(A_2)$ holds for operators $A$ depending locally with respect to time on $\bar{z}$:

\begin{cor}\label{C22}
Suppose that the embedding $E_1\hookrightarrow E_0$ is compact and let the operator $A$ be of the form
$
A[\bar{z}](t)=A_0\big(t,\bar{z}(t)\big)$,
where $A_0\in C_b^{\varrho,1-}\big(\R^+\times E_\sigma, \mathcal{H}(E_1,E_0)\big)$ for some $\varrho\in (0,1)$ and $\sigma\in [0,1)$. Then $A$ satisfies assumption $(A_2)$ for any $\alpha\in (\sigma,1]$. 
\end{cor}

\begin{proof}
It just remains to observe that the embedding $F_1:=E_\alpha\hookrightarrow E_\sigma=:F_0$ is compact according to \cite[I.Thm.2.11.1]{LQPP} since $\sigma<\alpha$ and due to the choice of admissible interpolation functors $(\cdot,\cdot)_\theta$.
\end{proof}

If $A$ is of the form \eqref{local}, then also the conditions \eqref{i}, \eqref{ii} for global existence are simpler to verify. Thus we consider again the unique maximal generalized solution $u=u(\cdot;u^0)\in C(J,\E_\beta)$  to \eqref{1}-\eqref{4} on $J=J(u^0)$ corresponding to $u^0$ as provided by Theorem \ref{T}. 

\begin{cor}\label{C222}
Suppose that the embedding $F_1\hookrightarrow F_0$ is compact and let the operator $A$ be of the form \eqref{local} with $A_0$ satisfying \eqref{A0}. Let $\Phi$ satisfy assumption $(A_5)$ and suppose
that for each $T>0$ there exist numbers $\rho\in (0,1)$ and $c_5(T)>0$ such that the solution $u$ to \eqref{1}-\eqref{4} satisfies
\bqn\label{u}
\left\| \Phi(\bar{u})(t)-\Phi(\bar{u})(t_*)\right\|_{F_1}\,\le c_5(T)\, \vert t-t_*\vert^\rho\ ,\quad t,t_*\in J_T:= J\cap [0,T]\ .
\eqn
Then \eqref{i} and \eqref{ii} hold.
\end{cor}

\begin{proof}
Since \eqref{u} in particular means that $\Phi(\bar{u})(J_T)$ is bounded in $F_1$, \eqref{ii} is immediate from \eqref{A0}. Analogously as in the proof of Lemma \ref{ex1}, condition \eqref{i} is a consequence of \cite[I.Cor.1.3.2]{LQPP} and the compact embedding $F_1\hookrightarrow F_0$.
\end{proof}

\subsection{Applications}

Since the following exemplary problems were studied elsewhere (except for the first one), we do not go too much into the details. Clearly, the results of Section \ref{main result} do not restrict to the examples presented herein.\\

For the remainder we fix a bounded subset $\Om$ of $\R^n$, $n\le 3$, with smooth boundary $\partial\Om$. Even though we may incorporate general time-dependent second order elliptic operators on $\Om$ subject to suitable boundary conditions, we restrict ourselves for the sake of simplicity to time-independent operators in divergence form, that is, operators of the form
\bqn\label{a3}
A_0(z)w\,:=\, -\nabla_x\cdot\big(D(z)\nabla_xw\big)
\eqn
subject to, e.g., Neumann conditions on $\partial\Om$. Here, the function $D$ is supposed to satisfy
\bqn\label{a4}
D\in C^{2-}(\R)\ ,\qquad D(z)\ge d_0>0\ ,\quad z\in \R\ ,
\eqn
for some number $d_0$. Introducing for $p\in (1,\infty)$ and $\theta\ge 0$ the Sobolev spaces (including Neumann boundary conditions)
\begin{equation*}
        \Wqb^{2\theta}:=\left\{ \begin{array}{ll} \big\{ u\in
            W_{p}^{2\theta}(\Om)\, ; \, \partial_\nu u=0\big\}\ , & 2\theta>1+1/p\ ,\\
            W_{p}^{2\theta}(\Om)\ ,& 0\leq 2\theta \le 1+1/p\ , \end{array} \right.
     \end{equation*}
we obtain that 
\bqn\label{a5}
\text{the embedding}\quad \Wqb^2\hookrightarrow L_p\quad \text{is compact}
\eqn
and
\bqn\label{a6}
E_{1/2}:=[L_p,\Wqb^2]_{1/2}=\Wqb^{1}\ ,\qquad E_\theta:= (L_p,\Wqb^2)_{\theta,p}=\Wqb^{2\theta}\ ,\quad 2\theta\in (0,2)\setminus\{1,1+1/p\}\ ,
\eqn
where the equality is up to equivalent norms and where $[\cdot,\cdot]_{1/2}$ and $(\cdot,\cdot)_{\theta,p}$ are the complex and real interpolation functors, respectively, all of which are admissible. Moreover, 
\bqn\label{a7}
A_0\in C_b^{1-}\big(C^1(\bar{\Om}),\mathcal{H}(\Wqb^2,L_p)\big)
\eqn
due to \eqref{a4}, and
\bqn\label{a77}
A_0(z)\quad \text{is resolvent positive for}\ z\in C^1(\bar{\Om})\ .
\eqn
We  assume that a non-negative function $b \in C^{2-}(\R^+\times\R^+\times\R)$ is given that satisfies, for any $T,R>0$,
\begin{align}
\vert D_3^k b(t,a,z)-D_3^k b(t,a,z_*)\vert\,&\le\, c(T,R)\, \g(a)\,\vert z-z_*\vert\ ,\label{a8a}
\end{align}
for $t\in [0,T]$, $a\ge 0$, $\vert z\vert, \vert z_*\vert\le R$, and $k=0,1$, where $\g\in L_{\infty,loc}(\R^+)$ satisfies \eqref{a1111}. For simplicity we also assume that $b$ is bounded, that is,
\bqn
b(t,a,z)\,\le c(T)\,\g(a)\ ,\label{a8b}
\eqn
for $t\in [0,T]$, $a\ge 0$, and all $z\in \R$ (this is merely needed to guarantee that solutions exists globally in the subsequent examples).
Thus, it follows from (the proof of) \cite[Lem.2.7]{WalkerEJAM} that
\begin{align}
\|  b(t,a,\bar{z})- b(t,a,\bar{z}_*)\|_{W_p^{2\bar{\alpha}}}\,&\le\, c(T,R)\, \g(a)\,\| \bar{z}-\bar{z}_*\|_{W_p^{2\alpha}}\ ,\label{a9a}\\
\| b(t,a,\bar{z})\|_{W_p^{2\bar{\alpha}}}\,&\le c(T,R)\,\g(a)\ ,\label{a9b}
\end{align}
for $t\in [0,T]$, $\bar{z},\bar{z}_* \in W_p^{2\alpha}$ with $\| \bar{z}\|_{W_p^{2\alpha}},\| \bar{z}_*\|_{W_p^{2\alpha}}\le R$, provided that $n/p<2\bar{\alpha}<2\alpha$. Also note that there is $2\mu\in (n/p,2\bar{\alpha})$ such that (see \cite{AmannMultiplication})
\bqn\label{a99}
\text{ pointwise multiplication}\quad 
W_p^{2\bar{\alpha}}\times W_p^{2\alpha}\rightarrow W_p^{2\mu}\quad\text{is continuous }\ .
\eqn
We put
$$
\Wqbb^{2\theta}:=L_1\big(\R^+,\Wqb^{2\theta},\g(a)\rd a\big)\ .
$$
Then $z\in\Wqbb^{2\theta}$ is non-negative if $z\in \Wqbb^{2\theta}\cap L_1\big(\R^+,L_p^+,\g(a)\rd a\big)$ with $L_p^+$ denoting the positive cone of $L_p=L_p(\Om)$.
Let 
\bqn\label{a10}
m\in C(\R^+)\quad \text{with}\quad m\ge 0\ .
\eqn

\subsubsection{Birth boundary conditions with delay}

We consider a model with history-dependent birth rate as investigated in \cite{DiBlasio} for the spatially homogeneous case:
\begin{align}
\partial_t u+\partial_a u&\,=\, \divv_x
\big(D(\bar{u}(t,x))\,\nabla_x u\big)\, -\, m(a)\, u\ , &
(t,a,x)\in
\R^+\times\R^+\times\Om\ ,\label{a11}\\
u(t,0,x)&\,=\, \int_0^\infty b\left(t,a,\int_{-\tau}^0\bar{u}(t+\sigma)\rd \sigma\right)\, u(t,a)\ \rd a \ ,
&
(t,x)\in\R^+\times\Om\ ,\label{a12}\\
u(s,a,x)&\,=\,F(s,a,x)\ ,&  (s,a,x)\in
[-\tau,0]\times \R^+\times\Om \ ,\label{a13}\\
\partial_\nu u(t,a,x)&\,=\,0\ ,& (t,a,x)\in
   \R^+\times \R^+\times \partial\Om\ ,\label{a14}\\
   \bar{u}(t,x)&\,=\, \int_0^\infty u(t,a,x)\ \rd a\ ,& (t,x)\in [-\tau,\infty)\times \Om\ ,\label{a15}
\end{align}
where $\tau>0$ is the maximal delay.

\begin{prop}\label{E1}
Let $\g \equiv 1$ and suppose \eqref{a3}, \eqref{a4}, \eqref{a8a}, \eqref{a8b}, \eqref{a10}. Let $F\in C([-\tau,0],\Wqbb^{2\beta})$ be non-negative, where $1+n/p<2\beta\le 2$. Then \eqref{a11}-\eqref{a15} admit a unique non-negative generalized solution $u\in C(\R^+,\Wqbb^{2\beta})$ with $\bar{u}\in C^1((0,\infty),L_p)\cap C((0,\infty),\Wqb^2)$.
\end{prop}

\begin{proof}
We merely sketch the proof. Extending a given function $u\in C(\R^+,\Wqbb^{2\beta})$ by
$$
u(t,\cdot,\cdot):=\left\{\begin{array}{ll}  u(t,\cdot,\cdot)\ , &t\ge 0\ ,\\
F(t,\cdot,\cdot)\  , &t\in [-\tau,0)\ ,\end{array}\right.
$$
and defining
$$
A[\bar{u}](t):=A_0(\bar{u}(t))\ ,\qquad B[u](t):= \int_0^\infty b\left(t,a,\int_{-\tau}^0\bar{u}(t+\sigma)\rd \sigma\right)\, u(t,a)\ \rd a\ ,
$$
equations \eqref{a11}-\eqref{a15} may be written in the form \eqref{1}-\eqref{4} with $u^0:=F(0,\cdot,\cdot)$. Then $(A_2)$ is a consequence of \eqref{a5}-\eqref{a7} and Corollary \ref{C22} by observing that 
\bqn\label{g}
\Wqb^{2\beta}\hookrightarrow \Wqb^{2\alpha}\hookrightarrow C^1(\bar{\Om})\ ,\quad 1+n/p<2\alpha<2\beta\ ,
\eqn
while $(A_3)$ follows from \eqref{a9a} and \eqref{a99}.
Therefore, local existence of a non-negative generalized maximal solution $u\in C(J,\Wqbb^{2\beta})$ on some maximal interval $J$ is immediate from Theorem \ref{T}, Proposition \ref{P2}, \eqref{a77}, and \eqref{a10}. Next note that $\bar{u}\in C^1(\dot{J},L_p)\cap C(\dot{J},\Wqb^2)$ solves
$$
\partial_t\bar{u}-\nabla_x\cdot\big(D(\bar{u})\nabla_x\bar{u}\big)=-\int_0^\infty m(a) u(t,a)\rd a+ B[u](t)=: f(t,x)
$$ 
in $\dot{J}_T\times\Om$, with $J_T:=J\cap [0,T]$ for $T>0$ fixed. Since $B[u]\in C(J_T,W_p^{2\mu})$ by \eqref{a99} and $W_p^{2\mu}\hookrightarrow C(\bar{\Om})$ we have $f\in C(J_T\times\bar{\Om})$. From \eqref{a8b} and the maximum principle we first obtain $\bar{u}\in L_\infty(J_T,L_\infty(\Om))$ and then $f\in BC(J_T\times\bar{\Om})$. Hence, \cite[Thm.4.2, Rem.4.3]{Amann93} entail that $\bar{u}:J_T\rightarrow C^{1+\epsilon}(\bar{\Om})$ is bounded and uniformly H\"older continuous with $\epsilon>0$. Since the embedding $F_1:=C^{1+\epsilon}(\bar{\Om})\hookrightarrow C^{1}(\bar{\Om})=:F_0$ is compact, we deduce \eqref{i} and \eqref{ii} from Corollary \ref{C222}, while \eqref{iii} is obvious. Proposition \ref{P3} then gives $J=\R^+$.
\end{proof}

\subsubsection{A tumor invasion model}

The following {\it haptotaxis} model describes the invasion of tumor cells (with density $u$) into the surrounding tissue along gradients of bound cell adhesion molecules (with density $f$) that are contained in the extracellular matrix. The cells produce a matrix degradative enzyme with density $v$. The model was studied in detail in \cite{WalkerDIE,WalkerEJAM}, and we just recall a very simple version:
\begin{align}
\partial_t u+\partial_a u&\,=\, \divv_x
\big(D(f)\,\nabla_x u-u\chi(f)\nabla_x f\big)\, -\, m(a)\, u\ , &
(t,a,x)\in
\R^+\times\R^+\times\Om\ ,\label{c1}\\
\partial_t f &\,=\, -vf\ ,&
(t,x)\in
\R^+\times\Om\ ,\label{c2}\\
\partial_tv&\,=\,\Delta_x v+\bar{u}-v\ ,&
(t,x)\in
\R^+\times\Om\ ,\label{c3}\\
u(t,0,x)&\,=\, \int_0^\infty b(t,a,\bar{u}(t))\, u(t,a)\ \rd a \ ,
&
(t,x)\in\R^+\times\Om\ ,\label{c4}\\
u(0,a,x)&\,=\, u^0(a,x)\ ,\ f(0,x)\,=\,f^0(x)\ ,\ v(0,x)\,=\,v^0(x)\ ,& (a,x)\in
 \R^+\times\Om \ ,\label{c5}\\
\partial_\nu v&\,=\, D(f)\partial_\nu u-u\chi(f)\partial_\nu f\, =\,0\ ,& (t,a,x)\in
   \R^+\times \R^+\times \partial\Om\ ,\label{c6}\\
   \bar{u}(t,x)&\,=\, \int_0^\infty u(t,a,x)\ \rd a\ ,& (t,x)\in
   \R^+\times \Om\ .\label{c7a}
\end{align}
If $\chi$ is smooth and $D$ satisfies \eqref{a4}, we obtain for
$$
A_1(f):=[w\mapsto w\chi(f)\nabla_x f]
$$
that
\bqn\label{c7b}
A_0 +A_1\in C_b^{1-}\big(\Wqb^2,\mathcal{H}(\Wqb^2,L_p)\big)\ ,\quad p>n\ .
\eqn
Given initial values $(f^0,v^0)\in \Wqb^{2+\tau}\times \Wqb^{2\tau}$ with $\tau>0$ and a suitable function $\bar{u}$, we first solve \eqref{c3} for $v$ and plug the result into equation \eqref{c2}. It follows from \cite[Lem.2.1]{WalkerSIAM} and \cite[Lem.2.6]{WalkerEJAM} that
\bqn\label{c8}
\Phi(\bar{u}):=f\quad \text{satisfies $(A_5)$ with $F_1:=\Wqb^{2+\epsilon}$, $\epsilon\in (0,\tau)$, and any $\alpha\in [0,1)$}\ .
\eqn
We then recall the result of \cite{WalkerEJAM}:

\begin{prop}\label{U1}
Let $\g\equiv 1$ and suppose \eqref{a4}, \eqref{a8a}, \eqref{a8b}, and \eqref{a10}. Let $\chi$ be a smooth function. Let $p>n$, $\tau>0$, and $2\beta\in (n/p,2)\setminus\{1+1/p\}$. Then, given non-negative initial values 
$$
(f^0,v^0,u^0)\in X:=\Wqb^{2+\tau}\times \Wqb^{2\tau}\times\Wqbb^{2\beta}
$$
there exists a unique non-negative solution $(f,v,u)\in C(\R^+,X)$ to \eqref{c1}-\eqref{c7a}, $f$ and $v$ being classical solutions to the corresponding equations. Moreover, $\bar{u}\in C^1(\R^+,L_p)\cap C(\R^+,\Wqb^2)$.
\end{prop}

\begin{proof}
We simply outline the main ideas of the proof of Proposition \ref{U1} and refer to \cite{WalkerEJAM} for details. First, local existence is immediate from Theorem \ref{T}, Corollary \ref{ex1}, \eqref{a9a}, \eqref{a10}, \eqref{a9b}, \eqref{c7b}, and \eqref{c8}. Given $T>0$ one can prove by a bootstrapping argument that $f=\Phi(\bar{u}):J_T\rightarrow \Wqb^2$ is uniformly H\"older continuous and bounded (see \cite[Eq.(3.26)]{WalkerEJAM}), whence \eqref{ii} follows from \eqref{c7b}. In particular, since $f(J_T)$ is bounded in $\Wqb^2$ and the embedding $\Wqb^2\hookrightarrow C^1(\bar{\Om})$ is compact, we derive from \cite[I.Cor.1.3.2]{LQPP} that $A_0(f(J_T))$ is a subset of $\mathcal{H}(\Wqb^2,L_p;\kappa,\omega)$ for some $\kappa\ge 1$, $\omega>0$. Considering $A_1(f)$ as a perturbation of $A_0(f)$, we deduce \eqref{i} using \cite[I.Thm.1.3.1(b)]{LQPP}. Thus $J=\R^+$ by Proposition \ref{P3} since \eqref{iii} is obvious.\\
\end{proof}

\subsubsection{Swarm-colony development of Proteus mirabilis}

Finally, we mention another example that fits into the abstract framework of \eqref{1}-\eqref{4}. The model describes the swarming phenomenon of a bacterium called {\it Proteus mirabilis}. It models the evolution of mononuclear ``swimmers'' with density $v$ and multi-cellular ``swarmers'' with density $u$ and reads
\begin{align}
\partial_t u+\partial_a u&\,=\, \divv_x
\big(D(\bar{u}(t,x))\,\nabla_x u\big)\, -\, m(a)\, u\ , &
(t,a,x)\in
\R^+\times\R^+\times\Om\ ,\label{d1}\\
\partial_t v&\,=\, \frac{1}{\tau}\big(1-\xi(v)\big)v+\int_0^\infty e^{a/\tau}\, m(a)\, u(t,a,x)\ \rd a\ ,& (t,x)\in \R^+\times\Om\ ,\label{d2}\\
u(t,0,x)&\,=\, \frac{1}{\tau}\,\xi\big(v(t,x)\big)\,v(t,x) \ ,
&
(t,x)\in\R^+\times\Om\ ,\label{d3}\\
u(0,a,x)&\,=\,u^0(a,x)\ ,\quad v(0,x)\,=\, v^0(x)\ ,&  (a,x)\in
 \R^+\times\Om \ ,\label{d4}\\
\partial_\nu u(t,a,x)&\,=\,0\ ,&(t,a,x)\in
   \R^+\times \R^+\times \partial\Om\ ,\label{d5}\\
   \bar{u}(t,x)&\,=\, \int_0^\infty u(t,a,x)\, e^{a/\tau}\ \rd a\ ,& (t,x)\in \R^+\times \Om\ ,\label{d6}
\end{align}
for some $\tau>0$. Let $\g(a):=\h(a):=e^{a/\tau}$. If $\xi$ is sufficiently smooth, $\mu$ is bounded, and $2\alpha\in (1+n/p,2)$, then
$$
B:=[u\mapsto \tau^{-1}\xi(v_u)v_u]\in C_b^{1-}\big(C([0,T],\Wqbb^{2\alpha}),C^1([0,T],\Wqb^{2\alpha})\big)
$$
satisfies $(A_3)$, where $v_u$ is for a given $u$ the solution to \eqref{d2} with $v^0\in\Wqb^2$. Moreover
$$B[u]\in C^1([0,T],L_p)\cap C([0,T],\Wqb^2)\ 
$$
if $u\in L_1([0,T],\Wqbb^2)$. Hence, we obtain from Theorem \ref{T} and Propositions \ref{P1}-\ref{P3}:

\begin{prop}
Suppose \eqref{a4}, \eqref{a10}, and let $m$ be bounded. Further let $\xi\in C^3(\R)$ and \mbox{$p>n$}. If $v^0\in\Wqb^2$ and $u^0\in \Wqbb^2\cap C^1(\R^+,L_p)\cap C(\R^+,\Wqb^2)$ are non-negative and satisfy \mbox{$\xi(v^0)v^0=\tau u^0(0,\cdot)$}, then there exists a unique non-negative solution 
$$
v\in C^1(\R^+,\Wqb^{2\alpha})\cap C(\R^+,\Wqb^2)\ ,\quad u\in C(\R^+,\Wqbb^{2\alpha})\cap L_{\infty,loc}(\R^+,\Wqbb^{2})\ ,\qquad \alpha\in (0,1)\ .
$$
Moreover, $u$ satisfies \eqref{r1}, \eqref{r2} with $E_0=L_p$.
\end{prop}

For details we refer to \cite{LaurencotWalkerOpus1}, in particular also for the (more realistic) case of degenerate diffusion.

%\section*{Acknowledgement}

\end{document}